\documentclass[reqno,11pt]{amsart}

\pdfoutput=1
\makeatletter
\let\origsection=\section \def\section{\@ifstar{\origsection*}{\mysection}}
\def\mysection{\@startsection{section}{1}\z@{.7\linespacing\@plus\linespacing}{.5\linespacing}{\normalfont\scshape\centering\S}}
\makeatother

\usepackage{amsmath,amssymb,amsthm}
\usepackage{mathrsfs}
\usepackage{mathabx}\changenotsign
\usepackage{dsfont}

\usepackage{xcolor}

\usepackage[backref]{hyperref}
\hypersetup{
    colorlinks,
    linkcolor={red!60!black},
    citecolor={green!60!black},
    urlcolor={blue!60!black}
}

\usepackage[open,openlevel=2,atend]{bookmark}

\usepackage[abbrev,backrefs]{amsrefs}

\usepackage[T1]{fontenc}
\usepackage{lmodern}
\usepackage[babel]{microtype}
\usepackage[english]{babel}

\usepackage{geometry}
\numberwithin{equation}{section}
\numberwithin{figure}{section}

\usepackage{enumitem}
\def\rmlabel{\upshape({\itshape \roman*\,})}

\def\alabel{\upshape({\itshape \alph*\,})}

\let\polishlcross=\l
\def\l{\ifmmode\ell\else\polishlcross\fi}

\def\tand{\ \text{and}\ }
\def\qand{\quad\text{and}\quad}
\def\qqand{\qquad\text{and}\qquad}

\let\emptyset=\varnothing
\let\setminus=\smallsetminus

\let\sm=\setminus

\makeatletter
\def\moverlay{\mathpalette\mov@rlay}
\def\mov@rlay#1#2{\leavevmode\vtop{   \baselineskip\z@skip \lineskiplimit-\maxdimen
   \ialign{\hfil$\m@th#1##$\hfil\cr#2\crcr}}}
\newcommand{\charfusion}[3][\mathord]{
    #1{\ifx#1\mathop\vphantom{#2}\fi
        \mathpalette\mov@rlay{#2\cr#3}
      }
    \ifx#1\mathop\expandafter\displaylimits\fi}
\makeatother

\newcommand{\dcup}{\charfusion[\mathbin]{\cup}{\cdot}}

\makeatletter
\DeclareFontFamily{U}  {MnSymbolC}{}
\DeclareSymbolFont{MnSyC}         {U}  {MnSymbolC}{m}{n}
\DeclareFontShape{U}{MnSymbolC}{m}{n}{
    <-6>  MnSymbolC5
   <6-7>  MnSymbolC6
   <7-8>  MnSymbolC7
   <8-9>  MnSymbolC8
   <9-10> MnSymbolC9
  <10-12> MnSymbolC10
  <12->   MnSymbolC12}{}
\DeclareMathSymbol{\powerset}{\mathord}{MnSyC}{180}
\makeatother

\makeatletter
\DeclareFontFamily{U}{MnSymbolA}{}
\DeclareFontShape{U}{MnSymbolA}{m}{n}{
    <-6>  MnSymbolA5
   <6-7>  MnSymbolA6
   <7-8>  MnSymbolA7
   <8-9>  MnSymbolA8
   <9-10> MnSymbolA9
  <10-12> MnSymbolA10
  <12->   MnSymbolA12}{}
\DeclareFontShape{U}{MnSymbolA}{b}{n}{
    <-6>  MnSymbolA-Bold5
   <6-7>  MnSymbolA-Bold6
   <7-8>  MnSymbolA-Bold7
   <8-9>  MnSymbolA-Bold8
   <9-10> MnSymbolA-Bold9
  <10-12> MnSymbolA-Bold10
  <12->   MnSymbolA-Bold12}{}
\DeclareSymbolFont{MnSyA}{U}{MnSymbolA}{m}{n}
\SetSymbolFont{MnSyA}{bold}{U}{MnSymbolA}{b}{n}

\DeclareRobustCommand{\overleftharpoon}{\mathpalette{\overarrow@\leftharpoonfill@}}
\DeclareRobustCommand{\overrightharpoon}{\mathpalette{\overarrow@\rightharpoonfill@}}
\def\leftharpoonfill@{\arrowfill@\leftharpoondown\mn@relbar\mn@relbar}
\def\rightharpoonfill@{\arrowfill@\mn@relbar\mn@relbar\rightharpoonup}

\DeclareMathSymbol{\leftharpoondown}{\mathrel}{MnSyA}{'112}
\DeclareMathSymbol{\rightharpoonup}{\mathrel}{MnSyA}{'100}
\DeclareMathSymbol{\mn@relbar}{\mathrel}{MnSyA}{'320}

\makeatother

\newtheorem{theorem}{Theorem}[section]

\newtheorem{lemma}[theorem]{Lemma}

\newtheorem{proposition}[theorem]{Proposition}

\newtheoremstyle{definition}
  {4pt}
  {4pt}
  {\sl}
  {}
  {\bfseries}
  {.}
  {.5em}
  {}

\theoremstyle{definition}
\newtheorem{definition}[theorem]{Definition}
\newtheorem{observation}[theorem]{Observation}

\theoremstyle{remark}

\newtheoremstyle{introthms}
  {3pt}
  {3pt}
  {\itshape}
  {}
  {\bfseries}
  {.}
  {.5em}
  {\thmnote{#3}}
\theoremstyle{introthms}

\let\eps=\varepsilon
\let\theta=\vartheta
\let\rho=\varrho
\let\phi=\varphi

\newcommand{\mbs}{\boldsymbol}

\def\N{\mathds N}

\def\NN{\mathds N}
\def\ZZ{\mathds Z}

\def\E{\mathds{E}}
\def\P{\mathds{P}}

\newcommand{\COV}{{\mbs{\mathrm{Cov}}}}

\def\cA{{\mathcal A}}
\def\cH{{\mathcal H}}

\def\cF{{\mathcal F}}
\def\cB{{\mathcal B}}
\def\cC{{\mathcal C}}

\def\cG{{\mathcal G}}
\def\cP{{\mathcal P}}
\def\cQ{{\mathcal Q}}

\def\cW{{\mathcal W}}
\newcommand{\cQj}{\mathcal{Q}_j}

\def\cR{{\mathcal R}}

\def\cK{{\mathcal K}}
\def\cU{{\mathcal U}}

\def\cW{{\mathcal W}}

\def\DISC{\mathrm{DISC}}

\def\CL{\mathrm{CL}}

\def\DISCQd{\ensuremath{\mathrm{DISC}_{\cQ,d}}}

\def\WDISCQd{\ensuremath{\mathrm{WDISC}_{\cQ,d}}}
\def\CLQd{\ensuremath{\mathrm{CL}_{\cQ,d}}}
\def\MINQd{\ensuremath{\mathrm{MIN}_{\cQ,d}}}
\def\DEVQd{\ensuremath{\mathrm{DEV}_{\cQ,d}}}

\def\DISCQdeps{\ensuremath{\mathrm{DISC}_{\cQ,d}(\eps)}}

\def\WDISCQdeps{\ensuremath{\mathrm{WDISC}_{\cQ,d}(\eps)}}
\def\CLQdFeps{\ensuremath{\mathrm{CL}_{\cQ,d}(F,\eps)}}
\def\MINQdeps{\ensuremath{\mathrm{MIN}_{\cQ,d}(\eps)}}
\def\DEVQdeps{\ensuremath{\mathrm{DEV}_{\cQ,d}(\eps)}}

\def\tand{\ \text{and}\ }
\def\bv{\boldsymbol{v}}

\def\by{\boldsymbol{y}}
\def\bu{\boldsymbol{u}}
\def\bz{\boldsymbol{z}}

\def\bf{\boldsymbol{f}}
\def\bell{\boldsymbol{\ell}}

\def\cl{{\rm db}}

\def\deg{{\rm deg}}
\def\indi{\mathds{1}}

\def\hom{{\rm hom}}

\def\Nat{\mathbb{N}}

\def\Ext{\mathrm{Ext}}

\newcommand{\VQ}[1]{V^{#1}}

\let\preceq=\preccurlyeq

\usepackage[single]{accents}
\newcommand{\arrow}[1]{\accentset{\rightharpoonup}{#1}}

\begin{document}
\title{Quasirandomness in hypergraphs}

\author[E. Aigner-Horev]{Elad Aigner-Horev}
\address{Department of Mathematics and Computer Science, Ariel University, Ariel, Israel}
\email{horev@ariel.ac.il}

\author[D. Conlon]{David Conlon}
\address{Mathematical Institute, University of Oxford, Oxford,
United Kingdom}
\email{David.Conlon@maths.ox.ac.uk}
\thanks{The second author was supported by a Royal Society University Research Fellowship and by ERC Starting Grant 676632.}

\author[H. H\`an]{Hi\d{\^e}p H\`an}
\address{Departamento de Matem\'atica y Ciencia de la Computaci\'on,
Universidad de Santiago de Chile, Santiago, Chile}
\email{hiep.han@usach.cl}
\thanks{The third author was supported by  the FONDECYT Iniciaci\'on grant  11150913 and  by Millenium Nucleus Information and Coordination in Networks.}

\author[Y. Person]{Yury Person}
\address{Institut f\"ur Mathematik, Goethe-Universit\"at,
  Frankfurt am Main, Germany}
\email{person@math.uni-frankfurt.de}
\thanks{The fourth author was supported by DFG grant PE 2299/1-1.}

\author[M. Schacht]{Mathias Schacht}
\address{Fachbereich Mathematik, Universit\"at Hamburg,
Hamburg, Germany}
\email{schacht@math.uni-hamburg.de}
\thanks{The fifth author was supported by ERC Consolidator Grant 724903.}

\date{}

\begin{abstract}
An $n$-vertex graph $G$ of edge density~$p$ is considered to be \emph{quasirandom} if it shares several important properties with
the random graph $G(n,p)$.
A well-known theorem of Chung, Graham and Wilson states that many such `typical' properties are asymptotically equivalent and, thus, a 
graph~$G$ possessing one such property automatically satisfies the others. 

In recent years, work in this area has focused on uncovering more quasirandom graph properties and on extending the known results to  
other discrete structures. In the context of hypergraphs, however, one may consider several different notions of quasirandomness. 
A complete description of these notions has been provided recently by Towsner, who proved several central equivalences using 
an analytic framework. We give short and purely combinatorial proofs of the main equivalences in Towsner's result. 
\end{abstract}
\maketitle


\section{Introduction}\label{sec:intro}

{\let\thefootnote\relax\footnote{{A strict subset of this work appeared in the EuroComb2017 conference proceedings as can be seen \href{https://www.sciencedirect.com/science/article/pii/S1571065317301002}{here}.}}}
Quasirandomness may be seen as the study of structures which share some of the typical properties of a random structure of the same size. 
This area has 
connections to and applications in several branches of pure mathematics and theoretical computer science. For further information, we 
refer the reader to the surveys~\cites{KSSSz02,KS06,Vad11}. We focus here on \emph{quasirandom graphs} and \emph{hypergraphs}.

Let $(G_n)_{n\in\NN}$ be a sequence of graphs, where $G_n$ has $n$ vertices. For a fixed $p\in[0,1]$, we say that $(G_n)_{n\in\NN}$ 
is $p$-quasirandom if the graphs $G_n$ have a \emph{uniform edge distribution} and density~$p$, that is, 
\begin{equation}\label{eq:uniform-edge-distribution}
	e(G_n[S]) = p \binom{|S|}{2} +o(n^2) \ \text{for every}\ S \subseteq V(G_n)\,,
\end{equation}
where $e(G_n[S])$ denotes the number of edges in the induced subgraph $G_n[S]$. The property above is often referred to as \emph{discrepancy}. 
Early results on quasirandom graphs implicitly appeared in~\cites{Al86,AlCh88,FRW88,Rodl86} and the systematic study was initiated by 
Thomason~\cites{Thomason-2,Thomason-1} and Chung, Graham and Wilson~\cite{CGW89}. The seminal result of Chung, Graham, and Wilson  
states that~\eqref{eq:uniform-edge-distribution} is a \emph{quasirandom property} in the sense that a sequence $(G_n)_{n\in\NN}$ satisfying  property~\eqref{eq:uniform-edge-distribution} will also satisfy several other properties typically expected (with high probability) of the random graph $G(n,p)$. For example, having uniform edge distribution is asymptotically equivalent to the property that
\begin{equation}\label{eq:cycles}
e(G_n)=p\binom{n}{2} +o(n^2) \qand N_{C_4}(G_n)=p^4n^4+o(n^4)\,,
\end{equation}
where $N_{C_4}(G_n)$ denotes the number of labeled copies of $C_4$, the cycle of length $4$, in $G_n$. This is somewhat surprising, as~\eqref{eq:cycles} seems at first glance to be a weaker condition. It is not difficult to show that \emph{any} graph $G_n$ on $n$ vertices with edge density $p$ contains at least $p^4n^4+o(n^4)$ labeled copies of $C_4$. Thus, a graph sequence $(G_n)_{n\in\NN}$ is quasirandom if and only if it is an asymptotic \emph{minimiser} for the number of copies of $C_4$.

Another quasirandom property  asserts that for every fixed graph $F$ we have
\begin{equation}\label{eq:copies}
N_{F}(G_n)=p^{e(F)}n^{v(F)}+o(n^{v(F)})\,,
\end{equation}
where again $N_{F}(G_n)$ denotes the number of labeled copies of $F$ and $v(F)$ and $e(F)$ denote the number of vertices and edges in $F$, respectively. There are also many other quasirandom properties for graphs besides those mentioned above (see, e.g.,~\cites{HL12,Janson,LS,ReiSch,Shapira,Shapira-Yuster,SimSos91,Simonovits-Sos-2,Skokan-Thoma,Yuster} and the references therein).

Besides quasirandom graphs notions of quasirandomness have been explored for other discrete structures, including 
hypergraphs~\cites{Ch90,ChuGra90hyper,HavTho89}, 
subsets of~$\ZZ/n\ZZ$~\cite{ChuGra92Zn}, 
set systems~\cite{ChuGra91sets}, 
tournaments~\cite{ChuGra91tour}, and groups~\cite{Gow08}. 
However, satisfactory generalisations to hypergraphs are surprisingly difficult to pin down. For example, R\"odl~\cite{Rodl86} observed that straightforward generalisations of~\eqref{eq:uniform-edge-distribution} and~\eqref{eq:copies} to hypergraphs are not equivalent, while a generalisation of~\eqref{eq:cycles} is anything but clear. 

More formally, let $(H_n)_{n\in\NN}$ be a sequence of $k$-uniform hypergraphs, i.e., pairs $(V_n,E_n)$ where the edge set~$E_n$ is a subset of all $k$-element subsets of $V_n$, which we denote by $\binom{V_n}{k}$, and suppose $|V_n| = n$. 
The straightforward generalisation of~\eqref{eq:copies} is
\begin{equation}\label{eq:copies_H}
N_{F}(H_n)=p^{e(F)}n^{v(F)}+o(n^{v(F)}) 
\end{equation}
for every fixed $k$-uniform hypergraph $F$, 
while the obvious generalisation of~\eqref{eq:uniform-edge-distribution} is
\begin{equation}\label{eq:disc}
e(H_n[S]) = p \binom{|S|}{k} +o(n^k)\ \text{for every}\ S \subseteq V(H_n)\,.
\end{equation}
However,~\eqref{eq:disc} does not imply~\eqref{eq:copies_H} when $k \geq 3$.
Instead, one needs to control the edges with respect to all $(k-1)$-uniform hypergraphs $G$ on the same vertex set. That is, we need to consider the property
\begin{equation}\label{eq:disc_strong}
e(H_n[G]) = p\,|\cK_k(G)| +o(n^k) \,\, \text{for every $(k-1)$-uniform } G \text{ on } V(H_n),
\end{equation}
where $e(H_n[G])$ denotes the number of edges $e$ of $H_n$ with $\binom{e}{k-1}\subseteq E(G)$ and $\cK_k(G)$ is the family of cliques on $k$ vertices that are contained in $G$. For $p = 1/2$, Chung and Graham~\cite{ChuGra90hyper} proved that~\eqref{eq:copies_H} and~\eqref{eq:disc_strong} are equivalent and that the correct generalisation of $C_4$ is the octahedron, i.e., the complete $k$-uniform $k$-partite hypergraph with classes of order $2$. Later, Kohayakawa, R\"odl and Skokan~\cite{KRS02} generalised this result to arbitrary fixed densities $p$.

More recently, it was shown by Kohayakawa, Nagle, R\"odl and Schacht~\cite{KNRS} that~\eqref{eq:disc} implies~\eqref{eq:copies_H} if one weakens the requirement of~\eqref{eq:copies_H} to counting linear (or simple) hypergraphs $F$, that is, hypergraphs where any two edges intersect in at most one vertex. As there are (weak) regularity lemmas for hypergraphs~\cites{Ch91,FrRo92,Steger} `compatible' with~\eqref{eq:disc}, this often allows one to use conceptually simpler tools for studying problems that involve linear hypergraphs only. The reverse implication, \eqref{eq:copies_H}\,$\Longrightarrow$\,\eqref{eq:disc}, was shown by Conlon, H\`an, Person and Schacht in~\cite{CHPS}, that is, provided~\eqref{eq:copies_H} holds for all linear hypergraphs $F$, then \eqref{eq:disc} also holds. The same authors also described several other such \emph{weakly quasirandom properties}, including an analogue of~\eqref{eq:cycles} where the r\^ole of $C_4$ is filled by an appropriate linear hypergraph (see~\cite{CHPS} for details). They also put forward a guess as to how one might introduce other discrepancy notions of intermediate strength and what the corresponding minimising hypergraphs should look like. Subsequently, Lenz and Mubayi~\cites{LM, LM15a, LM15b} extended the results of~\cite{CHPS} by adding a \emph{spectral property} and providing additional equivalences between certain notions of hypergraph quasirandomness of intermediate strength.

Finally, Towsner~\cite{Towsner} obtained a common generalisation of those earlier results on 
hypergraph quasirandomness, where the appropriate versions of~\eqref{eq:uniform-edge-distribution},~\eqref{eq:cycles}, and~\eqref{eq:copies} 
are equivalent. This he accomplished by using the language of non-standard analysis and hypergraph limits. By generalising constructions of Lenz and Mubayi~\cite{LM15b}, he also showed that these notions of quasirandomness are all distinct, again using analytic language. 
Towsner remarks that it would be of interest to finitise his arguments. Here we do just that, providing short combinatorial proofs for the main equivalences in  Towsner's work.

\section{Definitions and the main result} 

\subsection{Quasirandom properties for hypergraphs}
For a finite set $X$, we write $\arrow{X}$ to denote the set of all orderings of the members of $X$ and $\powerset(X)$ for its powerset. For an integer $k \geq 1$ and a set $V$, the set of all $k$-element subsets of $V$ is denoted by ${V \choose k}$ and we write~${V \choose k}_<$ to denote $\overrightharpoon{\binom{V}{k}}$. Given a set (of indices) $Q \subseteq [k]$, we write $V^Q$ for the set of all functions from $V$ to $Q$. Clearly~$V^Q$ is isomorphic to $V^{|Q|}$ 
and we refer to its members as $Q$-\emph{tuples}. Unlike the members of ${V \choose k}_<$, $Q$-tuples may contain non-distinct entries. By a $Q$-\emph{directed hypergraph}, we mean a pair $(V,E)$ where $E\subseteq V^Q$.
For a common generalisation of the `witness sets' in~\eqref{eq:disc} and~\eqref{eq:disc_strong} the following notation will be useful.

\begin{definition}
For $\cQ\subseteq \powerset([k])$, let $\cG=(G_Q)_{Q\in \cQ}$ be a sequence of $Q$-directed hypergraphs $G_Q$ on the same vertex set $V$. 
We say an ordered $k$-tuple $\bv=(v_1,\ldots,v_k)\in \binom{V}{k}_<$ is \emph{supported} by $\cG$ if, for every $Q\in\cQ$,
\[
	\bv_Q=(v_i\colon i\in Q)\in E(G_Q)\,.
\]
Moreover, we denote by $\cK_k(\cG) \subseteq {V \choose k}_<$ the set of all ordered $k$-tuples supported by $\cG$.
\end{definition}
Note that $\cK_k(\cG)=\binom{S}{k}_<$,  when we set $\cQ=\{\{1\},\dots,\{k\}\}=\binom{[k]}{1}$ and let $\cG$ consist of $k$ copies 
of the set~$S\subseteq V$ (viewed as a $1$-uniform hypergraph). Similarly, $\cK_k(\cG)=\overrightharpoon{\cK_k(G)}$ for
$\cQ=\binom{[k]}{k-1}$ and $\cG$ consists of $k$ copies of a $(k-1)$-uniform hypergraph $G$ indexed by the elements of $\cQ$.
In other words, by making appropriate choices for $\cQ$ we obtain (ordered) versions of the `witness sets' from~\eqref{eq:disc} 
and~\eqref{eq:disc_strong}. Considering ordered versions simplifies the presentation for families $\cQ$ which are not 
subfamilies of a level of the  Boolean lattice of subsets of $[k]$.
Below we define a version of discrepancy for hypergraphs for any family $\cQ\subseteq \powerset([k])$, 
which is the first quasirandom property we consider here. 

\begin{definition}[$\DISCQd$]
For an integer $k\geq 2$, a set system $\cQ\subseteq \powerset([k])$, and
reals $\eps>0$ and~$d\in[0,1]$, we say that a $k$-uniform hypergraph $H=(V,E)$
with $|V|=n$ satisfies $\DISCQdeps$ if, for every sequence $\cG=(G_Q)_{Q\in \cQ}$ of $Q$-directed hypergraphs
with vertex set~$V$,
\[
	\Big|\big|\arrow E\cap \cK_k(\cG)\big|-d\,\big|\cK_k(\cG)\big|\Big|\leq \eps n^k\,.
\]
\end{definition}

We also consider the following  weighted version of $\DISCQd$, where the sequence of directed hypergraphs 
$\cG$ is replaced by an ensemble of functions $\cW=\left(w_Q\colon V^Q\to[-1,1]\right)_{Q\in\cQ}$ and the set of supported $k$-tuples $\cK_k(\cG)$ is replaced with the function $\cW\colon V^{[k]}\to [-1,1]$ given by
\[
	\cW(\bv)=\prod_{Q\in\cQ}w_Q(\bv_Q),
\] 
where we set  $w_Q(\bv_Q)$ to be zero whenever  $\bv_Q$ is not a proper set, i.e., whenever it has any non-distinct entries.

\begin{definition}[$\WDISCQd$] \label{def:wdisc}
For an integer $k\geq 2$, a set system $\cQ\subseteq \powerset([k])$, and
reals $\eps>0$ and~$d\in[0,1]$, we say that a $k$-uniform hypergraph $H=(V,E)$
with $|V|=n$ satisfies $\WDISCQdeps$ if, for every ensemble of (weight) functions $\cW=(w_Q)_{Q\in \cQ}$
with $w_Q\colon \VQ{Q}\to[-1,1]$ for every $Q\in\cQ$,
\[
	\Bigg|\sum_{\bv\in V^{[k]}}\big(\indi_{\arrow E}(\bv)-d\big)\cW(\bv)\Bigg|\leq \eps n^k\,,
\]
where $\indi_{\arrow E}\colon V^{[k]}\to\{0,1\}$ denotes the indicator function of $\arrow E$.
\end{definition}

Letting $w_Q=\indi_{G_Q}$ for every $Q\in\cQ$, we note that the quantities 
$\sum_{\bv\in V^{[k]}}\big(\indi_{\arrow E}(\bv)-d\big)\cW(\bv)$ and $|\arrow E\cap \cK_k(\cG)|-d\,|\cK_k(\cG)|$ differ by 
$d$ times the number of $\bv\in V^{[k]}$ which have some non-distinct entries, yet are supported by $\cG$. However, this difference has order of magnitude $O_k(n^{k-1})$, so hypergraphs~$H$ satisfying $\WDISCQdeps$
must also satisfy $\DISCQd (2\eps)$ for sufficiently large $n$. The opposite implication follows by a simple averaging argument presented in Lemma~\ref{lem:disc_to_wdisc} below.

In the introduction, we noted that if a graph sequence $(G_n)_{n \in \mathbb{N}}$ with $|G_n| = n$ contains $d^{e(F)}n^{v(F)}+o(n^{v(F)})$ copies of each fixed graph $F$, then the sequence is $d$-quasirandom, that is, it satisfies the discrepancy condition~\eqref{eq:uniform-edge-distribution} with $p = d$. To state the  `counting' counterpart of $\DISCQd$ requires some notion of special hypergraphs. 

\begin{definition}[$\cQ$-simple]
\label{def:Qsimple}
We say that a $k$-uniform hypergraph $F=(V_F,E_F)$ is 
\emph{$\cQ$-simple} for a set system $\cQ\subseteq \powerset([k])$, 
if there is an ordering $E_F=\{f_1,\dots,f_m\}$ of its edges such that for every 
$i=1,\dots,m$ there is an ordering of the vertices of $f_i=\{x_{i_1},\dots,x_{i_k}\}$ with the property that for every 
$h<i$ there is a set $Q\in\cQ$ such that
\[
	\{r\colon x_{i_r}\in f_h\cap f_i\}\subseteq Q\,.
\]  
Here the orderings of the vertices for every edge of $F$ can be chosen independently and 
might not be compatible with each other. 
\end{definition}
It is easy to see that the notion of linear hypergraphs coincides with $\cQ$-simple hypergraphs for the set system $\cQ=\binom{[k]}{1}$,
while every $k$-uniform hypergraph is $\binom{[k]}{k-1}$-simple. The correct analogue of~\eqref{eq:copies_H} for hypergraphs 
having $\DISCQd$ is now the restriction to $\cQ$-simple hypergraphs~$F$ stated below.

\begin{definition}[$\CLQd$]
For an integer $k\geq 2$, a set system $\cQ\subseteq \powerset([k])$,
reals $\eps>0$, $d\in[0,1]$, and a $\cQ$-simple $k$-uniform hypergraph $F=(V_F,E_F)$,
we say that a $k$-uniform hypergraph $H=(V,E)$
with $|V|=n$ satisfies $\CLQdFeps$ if 
the number $N_F(H)$ of labeled copies of $F$ in $H$ satisfies
\[
	\Big|N_F(H)-d^{e(F)}n^{v(F)}\Big|\leq \eps n^{v(F)}\,.
\]
\end{definition}

Next we consider the appropriate generalisation of~\eqref{eq:cycles} for our setting.
Given a $k$-partite $k$-uniform hypergraph $F$ with vertex partition $V(F)=X_1\dcup\dots\dcup X_k$ and a set $Q\subseteq [k]$,
we define the \emph{$Q$-doubling of $F$} to be the hypergraph $\cl_Q(F)$
obtained by taking two copies of $F$ and identifying the vertex classes indexed by elements in $Q$. That is, the vertex set of the $Q$-doubling is 
\[
	V(\cl_Q(F))=Y_1\dcup\dots\dcup Y_k \quad\text{where}\quad Y_q=
		\begin{cases} 
			X_q\, & \text{if}\ q\in Q\,\\
			X_q\times \{0,1\}& \text{if}\ q\not\in Q\,
		\end{cases}
\]
and the edge set of the $Q$-doubling is the collection of all $k$-element sets of the form
$$
\{x_q\colon q\in Q\}\dcup\{(x_r,a)\colon r\in[k]\setminus Q\},$$
where $\{x_1,\dots,x_k\}\in E(F)$ and $a \in \{0,1\}$.
It is easy to check that for any two sets $Q$, $R\subseteq [k]$ and any $k$-partite $k$-uniform 
hypergraph~$F$ the ordering of the doubling operations does not matter, i.e.,
\[
	\cl_Q(\cl_R(F))=\cl_R(\cl_Q(F))\,.
\]
Hence, for $\cQ\subseteq \powerset([k])\setminus\{[k]\}$ (the operation $\cl_{[k]}$ leaves the hypergraph unchanged), we may define the $\cQ$-simple $k$-partite $k$-uniform hypergraph $M_{\cQ}$ 
recursively by setting
\[
	M_{\emptyset}=K_k^{(k)}\,,
\]
to be the $k$-partite $k$-uniform hypergraph consisting of one edge
and, for any $Q\in\cQ$, letting
\[
	M_{\cQ}=\cl_Q(M_{\cQ\setminus\{Q\}})\,.
\]
In the graph case $k=2$, we obtain $M_{\cQ}=C_4$ for $\cQ=\{\{1\},\{2\}\}$ and, for general $k\geq 2$, the hypergraphs $M_{\cQ}$ for 
$\cQ=\binom{[k]}{1}$ 
were shown to be minimisers for $\DISCQd$ in~\cite{CHPS}. Similarly, for $\cQ=\binom{[k]}{k-1}$, the hypergraphs~$M_{\cQ}$ are the $k$-uniform octahedra, i.e., complete 
$k$-partite $k$-uniform hypergraphs with vertex classes of size two, that appeared in the work of Chung and Graham~\cite{ChuGra90hyper}
and Kohayakawa, R\"odl, and Skokan~\cite{KRS02}.

It follows from these definitions that $M_{\cQ}$ consists of $2^{|\cQ|}$ hyperedges and 
$\sum_{i=1}^k2^{|\cQ|-\deg_{\cQ}(i)}$ vertices, where $\deg_{\cQ}(i)$ denotes the number of sets 
of $\cQ$ containing the element~$i$. An appropriate sequence of applications of the Cauchy--Schwarz inequality, one for 
each $Q\in\cQ$, shows that every $k$-uniform hypergraph $H$ on $n$ vertices with density $d>0$
contains at least $(d^{e(M_\cQ)}-o(1))n^{v(M_\cQ)}$ labeled copies of $M_{\cQ}$. The analogue of 
property~\eqref{eq:cycles} which we will show is equivalent to $\DISCQd$ is now as follows.

\begin{definition}[$\MINQd$]
For an integer $k\geq 2$, a set system $\cQ\subseteq \powerset([k])$, and
reals $\eps>0$ and $d\in[0,1]$,
we say that a $k$-uniform hypergraph $H=(V,E)$
with $|V|=n$ satisfies $\MINQdeps$ if 
\begin{enumerate}[label=\rmlabel]
	\item the density $d(H)=|E|/\binom{n}{k}$ satisfies $d(H)\geq d-\eps$ and
	\item the number $N_{M_{\cQ}}(H)$ of labeled copies of $M_{\cQ}$ in $H$ satisfies
	\[
		N_{M_{\cQ}}(H)\leq(d^{e(M_\cQ)}+\eps)n^{v(M_\cQ)}\,.
	\]
\end{enumerate}
\end{definition}

It is sometimes more convenient to work with the following  weighted version of $\MINQd$.

\begin{definition}[$\DEVQd$]
For an integer $k\geq 2$, a set system $\cQ\subseteq \powerset([k])$, and
reals $\eps>0$ and $d\in[0,1]$,
we say that a $k$-uniform hypergraph $H=(V,E)$
with $|V|=n$ satisfies $\DEVQdeps$ if 
\[
	\sum_M\prod_{f\in E(M)}(\indi_E(f)-d)\leq\eps n^{v(M_\cQ)}\,,
\]
where the sum ranges over all labeled copies $M$ of $M_\cQ$ in the complete $k$-uniform hypergraph~$K^{(k)}_V$ 
on the vertex set $V$. 
\end{definition}

\subsection{Main results}\label{sec:main}
For a property $P_{x_1,\dots,x_\ell}(\alpha_1,\dots,\alpha_r)$ of $k$-uniform hypergraphs, we say a {\sl sequence} of $k$-uniform hypergraphs $(H_n)_{n \in \Nat}$ satisfies $P_{x_1,\dots,x_\ell}$ if, for each choice of the parameters $\alpha_1,\dots,\alpha_r$ all but finitely many hypergraphs $H_n$ satisfy $P_{x_1,\dots,x_\ell}(\alpha_1,\dots,\alpha_r)$. Moreover, given two properties $P_{x_1,\dots,x_\ell}$ and $Q_{y_1,\dots,y_p}$, we say that $P_{x_1,\dots,x_\ell}$ \emph{implies} $Q_{y_1,\dots,y_p}$ and write 
\[
	P_{x_1,\dots,x_\ell} \implies Q_{y_1,\dots,y_p}
\]
if every sequence $(H_n)_{n \in \Nat}$ that satisfies $P_{x_1,\dots,x_\ell}$ also satisfies $Q_{y_1,\dots,y_p}$. 
Our main result is then the following. 

\begin{theorem}[Main result]\label{thm:main} For every $k \geq 2$, every set system $\cQ \subseteq \powerset([k])\sm \{[k]\}$, and $d \in [0,1]$, the properties $\DISCQd$, $\WDISCQd$, $\CLQd$, and $\DEVQd$ are all equivalent.
\end{theorem}

\noindent
In Section~\ref{sec:main-proof}, we will prove Theorem~\ref{thm:main} by establishing the chain of implications
\begin{equation}\label{eq:plan}
	 \DISCQd \Longrightarrow\WDISCQd \Longrightarrow \CLQd \Longrightarrow \DEVQd \Longrightarrow \WDISCQd  \Longrightarrow \DISCQd,
\end{equation}
where the last implication was already discussed after Definition~\ref{def:wdisc} above. One could also add $\MINQd$ to the list of equivalent properties in our  main result. Indeed, it is clear that $\CLQd \Longrightarrow \MINQd$. While the opposite implication also holds, we have chosen to omit the rather technical proof here. As well as the work of Towsner~\cite{Towsner}, where Theorem~\ref{thm:main} first appears, we refer the interested reader to~\cite{CJ}*{Lemma~5.8}, where the equivalence between $\DEVQd$ and $\WDISCQd$ is also proven as part of a broad spectrum of results about equivalences between different hypergraph norms.

While we will always work with general set systems, we follow Towsner in noting that antichains already capture the essence of the definitions above.
We briefly review this point. To begin, note that for any $k \geq 2$ and $\cQ \subseteq \powerset([k])$, there is a unique  antichain $\cA(\cQ) \subseteq \cQ$ with the property that 
\begin{equation}\label{eq:anti}
\mbox{for each $Q \in \cQ$ there exists $A \in \cA(\cQ)$ with $Q \subseteq A$.}
\end{equation}
In fact, $\cA(\cQ)$ consists of the inclusion maximal elements from $\cQ$. 
Note now, by~\eqref{eq:anti}, that the set of $\cA(\cQ)$-simple $k$-uniform hypergraphs coincides with the set of $\cQ$-simple $k$-uniform hypergraphs, so that $\CL_{\cA(\cQ),d}$ and $\CLQd$ define the same notion. Therefore, by Theorem~\ref{thm:main}, it follows that~$\cA(\cQ)$ and $\cQ$ define the same notion of quasirandomness.

\begin{observation}\label{obs:anti}
For every $k \geq 2$, $d \in [0,1]$, and $\cQ \subseteq \powerset([k])$, we have $\DISCQd \Longleftrightarrow \DISC_{\cA(\cQ),d}$.
\end{observation}

Observation~\ref{obs:anti} is in fact a special case of a broader principle. Given two set systems $\cA$, $\cB \subseteq \powerset([k])$, write $\cA \preceq \cB$ if there exists a bijection $\varphi\colon [k] \to [k]$ such that for every $A \in \cA$ the set~$\varphi(A)= \{\varphi(a)\colon a\in A\}$ is contained in the downset generated by $\cB$. Note that if $\cA \preceq \cB$ then the $\cA$-simple $k$-uniform hypergraphs are a subset of the $\cB$-simple $k$-uniform hypergraphs. This then yields the following observation.

\begin{observation}\label{obs:anti2}
For every $k \geq 2$, $d \in [0,1]$, and $\cA, \cB \subseteq \powerset([k])$ with $\cA \preceq \cB$, we have $\DISC_{\cB,d} \Longrightarrow \DISC_{\cA,d}$. 
\end{observation}
As previously mentioned, Towsner~\cite{Towsner}*{Section~9}, generalising ideas of Lenz and Mubayi~\cite{LM15b}, provided constructions of hypergraphs that distinguish the various notions of hypergraph quasi\-randomness defined above. We do the same. 
Our construction is essentially that of Towsner, with the distinction between Towsner's work and ours being in the analysis of the construction. In particular, our approach uses only some simple applications of the Chernoff and Chebyshev inequalities. 

For a simpler presentation we focus on the special case of distinguishing $\DISC_{\cQ,1/2}$ 
from $\DISC_{\cU,1/2}$, where both $\cQ$, $\cU \subseteq \binom{[k]}{i}$ are comprised only of $i$-sets for some $1 \leq i <k$ and 
$\cU \subsetneq \cQ$.
The analysis for densities other than $1/2$ and for more general set systems $\cQ$ and $\cU$ follows along similar lines, but would require somewhat more technical notation.

\begin{proposition}\label{lem:construction}
For every $1 \leq i < k$ and $\cU\subsetneq\cQ\subseteq \binom{[k]}{i}$ there exists $\delta > 0$ such that for every $\eps > 0 $ 
there is a sequence of hypergraphs $\cH = (H_n)_{n \in \N}$ which satisfies $\DISC_{\cU,1/2}(\eps)$, but fails to satisfy $\DISC_{\cQ,1/2}(\delta)$.
\end{proposition}

We present the proof of Proposition~\ref{lem:construction} in Section~\ref{sec:distinguish} and in the next section we give the details of the proof of Theorem~\ref{thm:main}.

\section{Equivalences of quasirandom properties}\label{sec:main-proof}

In this section, we prove Theorem~\ref{thm:main} by following the plan set out in~\eqref{eq:plan}. 

\subsection{\texorpdfstring{$\mbs{\DISCQd \Longrightarrow \WDISCQd}$}{DISC implies WDISC}}
Our proof of the implication $\DISCQd\implies\WDISCQd$ is an adaptation of an argument of Gowers~\cite{Gowers06}*{Section~3}. 

\begin{lemma}\label{lem:disc_to_wdisc}
For every $k \geq 2$, every set system $\cQ \subseteq \powerset([k]) \setminus \{[k]\}$, every $d \in [0,1]$, and every $\delta>0$, there exists an $\eps >0$ such that, for all sufficiently large $n$, if $H = (V,E)$ is an $n$-vertex $k$-uniform hypergraph satisfying $\DISCQd(\eps)$, then $H$ satisfies $\WDISCQd(\delta)$. 
\end{lemma}

\begin{proof}
Given $k$, $d$, $\delta$ and $\cQ = \{Q_1,\ldots,Q_\ell\}$, we set
\begin{equation}\label{eq:disc->wdisc}
\eps = \frac{\delta}{2^{|\cQ|+1}}\,.
\end{equation}
Let $H = (V,E)$ be an $n$-vertex $k$-uniform hypergraph satisfying $\DISCQd(\eps)$ and assume, for the sake of contradiction, that $H$ does not satisfy $\WDISCQd(\delta)$. 
Then there exists a collection of functions $\left( w_Q \colon \VQ{Q} \to [-1,1]\right)_{Q \in \cQ}$ such that 
\[
\left| \sum_{\bv \in V^{[k]}}\left(\indi_{\arrow E}(\bv) -d \right)\prod_{Q \in \cQ} w_Q(\bv_Q)\right|  >  \delta n^k\,.
\]
By writing $w_Q = w_Q^+ - w_Q^-$ for each $Q \in \cQ$, where $w_Q^+$ and $w_Q^-$ are both of the form $V^Q \to [0,1]$, we see that there are $|\cQ|$ functions $s_1,\ldots,s_{\ell}$ with $s_i\in\{w_{Q_{i}}^+,w_{Q_{i}}^-\}$ for every $i \in [\ell]$, such that
\begin{equation}\label{eq:expectation}
\left| \sum_{\bv \in V^{[k]}}\left(\indi_{\arrow E}(\bv) -d \right)\prod_{i=1}^{\ell}s_i(\bv_{Q_i})\right|  > 2^{-|\cQ|} \delta n^k \overset{\eqref{eq:disc->wdisc}}{=} 2\eps n^k.
\end{equation}

Let $\cF=(F_Q)_{Q\in\cQ} = (F_{Q_i})_{i \in [\ell]}$ be the family of random directed hypergraphs where $F_{Q_i}$ is the random $Q_{i}$-directed hypergraph where every possible edge $\bf \in V^{Q_i}$ is placed in $F_{Q_i}$ with probability $s_i(\bf)$ (as usual, we take $s_i(\bf)=0$ if $\bf$ has some identical entries).   
Let $U \subseteq V^{[k]}$ denote the random subset of $V^{[k]}$ where $\bv$ is in $U$ if the set $\bv_Q\in E(F_Q)$ for all $Q\in \cQ$. By the definition of~$\cF$, the probability that $\bv$ is in $U$ is given by $\prod_{i=1}^{\ell}s_i(\bv_{Q_i})$. The left-hand side of~\eqref{eq:expectation} under the absolute value is then the expectation of the random variable $X=\sum_{\bv\in U}\big(\indi_{\arrow E}(\bv) -d \big)$.  Therefore, by~\eqref{eq:expectation}, there is a choice of set $\tilde{U}$ for which
\[
\bigg| \sum_{\bv \in \tilde{U}}\left(\indi_{\arrow E}(\bv) -d \right)\bigg|  >  2\eps n^k\,.
\]
Suppose now that $\cG=(G_Q)_{Q\in \cQ}$ is the family of directed hypergraphs from which $\tilde{U}$ is derived, that is, $\tilde{U}$ consists exactly of those $\bv$ such that $\bv_Q\in E(G_Q)$ for every $Q \in \cQ$. Then  $\cK_k(\cG) \subseteq \tilde{U}$ and $\tilde{U}\setminus \cK_k(\cG)$ contains only $k$-tuples whose entries are not distinct. Since $|\tilde{U}\setminus \cK_k(\cG)|=O_k(n^{k-1})$, we see that, for $n$ sufficiently large,
\begin{align*}
\Big|\big|\arrow E\cap \cK_k(\cG)\big|-d\,\big|\cK_k(\cG)\big|\Big| 
&= 
\bigg| \sum_{\bv \in \cK_k(\cG)}\big(\indi_{\arrow E}(\bv) -d \big)\bigg| \\
& = 
\bigg| \sum_{\bv \in \tilde{U}}\big(\indi_{\arrow E}(\bv) -d \big)\bigg| -O_k(n^{k-1})
 >  2\eps n^k-O_k(n^{k-1}) > \eps n^k\,,
\end{align*}
which contradicts our assumption that $H$ satisfies $\DISCQd(\eps)$.
\end{proof}

\subsection{\texorpdfstring{$\mbs{\WDISCQd \Longrightarrow \CLQd}$}{WDISC implies CL}}

The following lemma shows that $\WDISCQd$ yields the appropriate counting result for $\cQ$-simple hypergraphs~$F$.

\begin{lemma}\label{lem:wdisc_to_cl}
For every $k \geq 2$, every set system $\cQ \subseteq \powerset([k]) \setminus \{[k]\}$, every $d \in [0,1]$, every $\cQ$-simple $k$-uniform hypergraph $F$, and every $\delta>0$, there exists an $\eps >0$ such that, for all sufficiently large~$n$, if $H = (V,E)$ is an $n$-vertex $k$-uniform hypergraph satisfying $\WDISCQd(\eps)$, then $H$ satisfies~$\CLQd(F,\delta)$. 
\end{lemma}

\begin{proof}
Given $k, \cQ, d, F$, and $\delta$, we set 
\[
	\eps = \frac{\delta/2}{(2^{e(F)}-1)}\qqand
	\eps' = \frac{\delta}{2}
\] 
and write $\hom(F,H)$ for the number of homomorphisms from $F$ to $H$. Note that $N_F(H)$, which is the number of injective homomorphisms, satisfies 
\[
	N_F(H)\leq \hom(F,H) \leq  N_F(H) + \eps' n^{v(F)}
\] 
for sufficiently large $n$. Indeed, there are at most $O_{v(F)}(n^{v(F)-1})$ non-injective homomorphisms from $F$ to  $H$ and this is at most $\eps' n^{v(F)}$ for $n$ sufficiently large. It will therefore suffice to prove that 
\begin{equation}\label{eq:homo-bound}
\hom(F,H) = d^{e(F)}n^{v(F)} \pm \eps' n^{v(F)}. 
\end{equation}
We  have 
\begin{equation}\label{eq:hom_trick}
	\hom(F,H) 
	= \sum_{\varphi\colon  V(F) \to V} \prod_{f \in E(F)} \indi_E(\varphi(f))
	= \sum_{\varphi\colon  V(F) \to V} \prod_{f \in E(F)} \left(\indi_E(\varphi(f))-d+d\right),
\end{equation}
where here the sum ranges over all functions $V(F) \to V$ and not just over homomorphisms. For $e \in E(H)$, write $g(e)=\indi_E(e)-d$. Multiplying out the expression $\prod_{f \in E(F)} (g(\varphi(f)) +d)$, we obtain~$2^{e(F)}$ summands, one corresponding to each subhypergraph of $F$. These summands have the form $\big(\prod_{f \in E(F')} g(\varphi(f))\big)d^{e(F)-e(F')}$ for some subhypergraph $F' \subseteq F$. In particular, when $F'$ is empty, the corresponding summand is $d^{e(F)}$. We may therefore rewrite~\eqref{eq:hom_trick} as
\begin{equation}\label{eq:many_terms}
\hom(F,H) = d^{e(F)}n^{v(F)} + \sum_{\substack{F'\subseteq F\\ e(F')\ge 1}}d^{e(F)-e(F')}\sum_{\varphi\colon  V(F) \to V} \prod_{f \in E(F')} g(\varphi(f)).
\end{equation}
 
We will argue that each of the sums 
\begin{equation}\label{eq:inner_sum}
\sum_{\varphi\colon V(F) \to V} \prod_{f \in E(F')} g(\varphi(f))
\end{equation}
is \emph{small}. To make this precise, let $F'$ be fixed and let $\{f_1,\ldots,f_{e(F')}\}$ be an ordering of the edges of $F'$ which certifies its $\cQ$-simplicity. Let $f'$ denote $f_{e(F')}$, the last edge in this ordering, and let $x_1,\ldots,x_k$ be the vertices of the edge $f'$, again ordered so as to certify $\cQ$-simplicity (see Definition~\ref{def:Qsimple}). We may rewrite~\eqref{eq:inner_sum} as
\begin{equation}\label{eq:split_sum}
	\sum_{\varphi\colon  V(F) \to V} \prod_{f \in E(F')} g(\varphi(f))
	=
	\sum_{\varphi'\colon  V(F)\setminus f' \to V}\sum_{\substack{\varphi\colon  V(F) \to V\\\varphi|_{V(F)\setminus f'} \equiv \varphi'}} \prod_{f \in E(F')} g(\varphi(f))
\end{equation}
and, for each (fixed) $\varphi'$, we may further rewrite the inner sum in~\eqref{eq:split_sum} as
\begin{equation}\label{eq:hom_wdisc}
	\sum_{\substack{\varphi\colon  V(F) \to V\\\varphi|_{V(F)\setminus f'} \equiv \varphi'}}\prod_{f \in E(F')} g(\varphi(f))
	=
	\sum_{\bv=(v_1,\dots,v_k)\in V^{[k]}}
	\sum_{\substack{\varphi\colon  V(F) \to V\\
 		  \varphi(x_i)=v_i \forall i\in[k] \\
		  \varphi|_{V(F)\setminus f'} \equiv \varphi'}}  g(\varphi(f')) \prod_{f \in E(F')\setminus\{f'\}} g(\varphi(f))\,.
\end{equation}

Finally, we explain how one may apply $\WDISCQd(\eps)$ to estimate the right-hand side of~\eqref{eq:hom_wdisc}.    
By $\cQ$-simplicity, for every $f\in E(F')\setminus\{f'\}$ there exists a set $Q\in\cQ$ with $\{i\colon x_i\in f\}\subseteq Q$. Therefore, there exists a partition of $E(F')\setminus\{f'\}$ into (possibly empty) sets $(E_Q)_{Q\in \cQ}$ such that for every $Q\in\cQ$ and $f\in E_Q$, we have $\{i\colon x_i\in f\}\subseteq Q$. For $f \in E(F')$, let us write $I_f=\{i\colon x_i\in f \cap f'\}$ to denote the indices of the elements appearing in $f \cap f'$, noting that $\bigcup_{f\in E_Q}I_f\subseteq Q$ for all $Q\in\cQ$.

For any $f \in E(F)$, $\varphi(f)$ is composed of two parts: the images of the vertices in $f \cap f' \subseteq \{x_1,\ldots,x_k\}$ and the images of the vertices in $f \sm f'$. In~\eqref{eq:hom_wdisc}, the images of these latter vertices are already fixed by $\varphi'$. With this in mind, we define functions 
$\left(w_Q\colon  \VQ{Q}\to[-1,1]\right)_{Q \in \cQ}$ by
\begin{equation}\label{eq:define_W}
	w_Q\left(\by\right)
	=
	\prod_{f\in E_Q} \Big(\indi_{E}\big(\{y_i\colon i\in I_f\}\cup\varphi'(f\setminus f')\big)-d\Big)\,. 
\end{equation}
That is, using $\by \in V^Q$ we pick images $\{y_i\colon i\in I_f\}$ for the elements $x_i$ appearing in the indices specified by $I_f$. Hence, if $\varphi$ is the extension of $\varphi'$ given by taking $y_i=\varphi(x_i)$ for all $i \in \bigcup_{f\in E_Q}I_f \subseteq Q$, the right-hand side of~\eqref{eq:define_W} corresponds exactly to $\prod_{f \in E_Q} g(\varphi(f))$.


Therefore, since, for any vector $\bz=(z_1,\ldots,z_k)\in V^{[k]}$, we have 
\[
	g(\bz)=g(\{z_1,\ldots,z_k\})=\indi_E(\{z_1,\ldots,z_k\})-d=\indi_{\arrow E}(\bz)-d\,,
\] 
we may rewrite the right-hand side of~\eqref{eq:hom_wdisc} as 
\[
	\sum_{\bv=(v_1,\dots,v_k)\in V^{[k]}}
	\sum_{\substack{\varphi\colon  V(F) \to V\\
 		  \varphi(x_i)=v_i \forall i\in[k] \\
		  \varphi|_{V(F)\setminus f'} \equiv \varphi'}}  g(\varphi(f')) \prod_{f \in E(F')\setminus\{f'\}} g(\varphi(f))
	=
	\sum_{\bv\in V^{[k]}} \left(\indi_{\arrow E}(\bv)-d\right)\prod_{Q\in\cQ} w_Q(\bv_Q)\,.
\]
By $\WDISCQdeps$, the right-hand side of the identity above is at most $\eps n^k$ in absolute value. Thus, we may bound~\eqref{eq:split_sum} (which is also~\eqref{eq:inner_sum}) by $\eps n^{v(F)}$. This in turn allows us to write~\eqref{eq:many_terms} as
\[
\hom(F,H) = d^{e(F)}n^{v(F)}\pm \big(2^{e(F)}-1\big)\eps n^{v(F)}=d^{e(F)}n^{v(F)}\pm\frac{\delta}{2} n^{v(F)},
\]
which completes the proof of~\eqref{eq:homo-bound}.
\end{proof}

\subsection{\texorpdfstring{$\mbs{\CLQd \Longrightarrow \DEVQd}$}{CL implies DEV}}
Recall that $N_F(H)$ denotes the number of labeled copies of $F$ in $H$. We also write $N^*_{F',F}(H)$ for the number of labeled copies of $F'$ that 
are \emph{induced with respect to~$F$ in~$H$}, that is, the number of injections $\varphi\colon V(F)\to V(H)$ such that for all $f\in E(F)$ we have $\varphi(f)\in E(H)$ if and only if $f\in E(F')$. The following lemma, whose proof by the principle of inclusion and exclusion follows verbatim from Facts~8 and~9 in~\cite{CHPS}, provides the required implication. We include its short proof for completeness.
\begin{lemma}\label{lem:cl_to_dev}
For every $k \geq 2$, every set system $\cQ\subset \powerset([k])\setminus\{[k]\}$, every $d \in [0,1]$, and every $\delta> 0$, 
there exists an $\eps>0$ such that if $H =(V,E)$ is an $n$-vertex $k$-uniform hypergraph that satisfies $\CLQdFeps$ for all $F\subseteq M_\cQ$, then $H$ satisfies $\DEVQd(\delta)$.
\end{lemma}
\begin{proof}
We shall bound $\sum_{M}\prod_{f\in E(M)}  (\indi_E(f)-d)$ with $M$ running over all copies of $M_{\cQ}$ in the complete 
hypergraph $K^{(k)}_V$ on the vertex set $V$.
By the inclusion-exclusion principle, we have for every spanning $F'\subseteq M_\cQ$
\[
	N^*_{F',M_\cQ}(H)=\sum_{F'\subseteq  F\subseteq M_\cQ} (-1)^{e(F)-e(F')}N_F(H)\,.
\]
Since $\CLQdFeps$ holds for all $F\subseteq M_\cQ$, we see that
\begin{align*}
	\sum_{M}\prod_{f\in E(M)}  (\indi_E(f)-d) 
	&= 
	\sum_{F'\subseteq M_\cQ} N^*_{F',M_\cQ}(H) (1-d)^{e(F')}(-d)^{e(M_\cQ)-e(F')} \\
	&= 
	\sum_{F'\subseteq M_\cQ}(1-d)^{e(F')}(-d)^{e(M_\cQ)-e(F')}\sum_{F'\subseteq  F\subseteq M_\cQ} (-1)^{e(F)-e(F')}N_F(H) \\
	&\leq 
	\left|\sum_{F'\subseteq M_\cQ}(1-d)^{e(F')}(-d)^{e(M_\cQ)-e(F')}\sum_{F'\subseteq  F\subseteq M_\cQ} (-1)^{e(F)-e(F')} d^{e(F)}n^{v(M_\cQ)} \right|\\
	&\hspace{2cm}+2^{2e(M_\cQ)}\eps n^{v(M_Q)} \\
	& = 
	\delta n^{v(M_\cQ)}\,,
\end{align*}
where we chose $\eps=\delta/2^{2e(M_\cQ)}$ and used the binomial theorem to show that
\begin{align*}
\sum_{F'\subseteq M_\cQ} (1-d)^{e(F')}&(-d)^{e(M_\cQ)-e(F')}\sum_{F'\subseteq  F\subseteq M_\cQ} (-1)^{e(F)-e(F')} d^{e(F)}\\
& = \sum_{F'\subseteq M_\cQ} (1-d)^{e(F')}(-d)^{e(M_\cQ)-e(F')}  d^{e(F')}\sum_{F'\subseteq  F\subseteq M_\cQ} (-d)^{e(F)-e(F')} \\
& = \sum_{F'\subseteq M_\cQ} (1-d)^{e(F')}(-d)^{e(M_\cQ)-e(F')}  d^{e(F')} (1-d)^{e(M_\cQ)-e(F')}\\
& = (1-d)^{e(M_\cQ)}\sum_{F'\subseteq M_\cQ} (-d)^{e(M_\cQ)-e(F')}  d^{e(F')}\\
& =0\,.\qedhere
\end{align*}
\end{proof}

\subsection{\texorpdfstring{$\mbs{\DEVQd \Longrightarrow \WDISCQd}$}{DEV implies WDISC}}
Recall that $M_\cQ$ (for some $\cQ \subset \powerset([k])$) is the $k$-uniform hypergraph obtained from a sequence of doubling operations. 
Assume that  $\cQ \subset \powerset([k])$ consists of $\ell$ sets~$Q_1$,\ldots,$Q_{\ell}$ for some ordering of the sets of $\cQ$.
We set $\cQ_j=\{Q_1,\ldots,Q_j\}$ and let $M_{\cQj}$ be the subhypergraph of~$M_\cQ$ formed by the $j$ doublings around 
$Q_1, \dots, Q_j$. We also set $M_{\cQ_0}=M_{\emptyset}=K^{(k)}_k$. Given any $k$-partite $k$-uniform hypergraph $M$, we refer to the 
$j$-th vertex class of $M$ by $V_j(M)$ and we write $V_Q(M)=\bigcup_{j\in Q}V_j(M)$ for any $Q\subseteq [k]$. 

The implication $\DEVQd \Longrightarrow \WDISCQd$ is a consequence of the following lemma.

\begin{lemma}\label{lem:dev_to_wdisc}
For every $k \geq 2$, every set system $\cQ=\{Q_1,\ldots,Q_\ell\} \subset \powerset([k])\setminus\{[k]\}$, every $d \in [0,1]$, and every $\delta> 0$, 
there exists an $\eps>0$ such that, for all sufficiently large $n$, if $H =(V,E)$ is an $n$-vertex $k$-uniform hypergraph that satisfies
\begin{equation}\label{eq:wdev_}
\Bigg|\sum_{\varphi\colon V(M_\cQ)\to V}\prod_{f\in E(M_\cQ)} \left(\indi_{\arrow E}\big(\varphi(f)\big)-d\right)\Bigg|\le  \eps n^{v(M_\cQ)}\,,
\end{equation}
then $H$ satisfies $\mathrm{WDISC}_{\cQ,d}(\delta)$. 
\end{lemma}

It is easy to see that~\eqref{eq:wdev_} is equivalent to $\DEVQd$ since all but $O_k(n^{v(M_\cQ)-1})$ functions $\varphi$ are injective and thus correspond to labeled copies of $M_\cQ$ in the complete $k$-uniform hypergraph on $V$. Moreover, since the doubling $\cl_{[k]}$ leaves the $k$-uniform hypergraph unchanged, taking $[k]\not\in \cQ$ is not a restriction.

\begin{proof}[Proof of Lemma~\ref{lem:dev_to_wdisc}]
Let $\cW = \left(\omega_Q\colon \VQ{Q} \to [-1,1]\right)_{Q \in \cQ}$ be any collection of weight functions. With $V(M_\emptyset)=[k]$, we write
\begin{equation}  \label{eq:fdisc_cauchy}
	\Bigg|\sum_{\bv\in V^{[k]}}\big(\indi_{\arrow E}(\bv)-d\big)\cW(\bv)\Bigg|^{2^\ell}
	= 
	\Bigg|\sum_{\varphi\colon V(M_\emptyset)\to V}\big(\indi_{\arrow E}(\varphi(1),\ldots,\varphi(k))-d\big)\cW\big(\varphi(1),\ldots,\varphi(k)\big)\Bigg|^{2^\ell}.
\end{equation}
We shall apply the Cauchy--Schwarz inequality $\ell$ times to~\eqref{eq:fdisc_cauchy}, each time separating a function~$\omega_Q$ (using the fact that $0\le \omega_Q^2\le 1$). Recalling that for $Q\subseteq [k]$ and $f=(x_1,\dots,x_k)$, $f_Q=(x_i\colon i\in Q)$, below we will show that 
for each $j = 0,\ldots,\ell-1$ we have
\begin{multline}\label{eq:CS_step}
	\Bigg|\sum_{\varphi\colon V(M_{\cQ_j})\to V}\prod_{f\in E(M_{\cQ_j})}\Big(\indi_{\arrow E}\big(\varphi(f)\big)-d\Big)
	\bigg(\prod_{Q\in\cQ\setminus\cQ_j} \omega_Q\big(\varphi(f_Q)\big)\bigg)\Bigg|^{2}\\
	\leq 
	n^{|V_{Q_{j+1}}(M_{\cQ_j})|}\cdot 
	\Bigg|\sum_{\varphi\colon V(M_{\cQ_{j+1}})\to V}\prod_{f\in E(M_{\cQ_{j+1}})}\Big(\indi_{\arrow E}\big(\varphi(f)\big)-d\Big)
	\bigg(\prod_{Q\in\cQ\setminus\cQ_{j+1}} \omega_Q\big(\varphi(f_Q)\big)\bigg)\Bigg|\,.
\end{multline}
In fact, to see~\eqref{eq:CS_step}, we rewrite the sum on the left-hand side of~\eqref{eq:CS_step} as a double sum in which the 
first sum is over all $\psi\colon V_{Q_{j+1}}(M_{\cQ_j})\to V$ and the second sum is over all extensions of $\psi$ to 
$\varphi\colon V(M_{\cQ_j})\to V$. Since $\varphi$ extends $\psi$ we have 
$\omega_{Q_{j+1}}(\varphi(f_{Q_{j+1}}))=\omega_{Q_{j+1}}(\psi(f_{Q_{j+1}}))$, where we view the edge $f\in E(M_{\cQ_j})$ 
as an ordered $k$-tuple (according to the $k$ vertex classes of $M_{\cQ_j}$), $f_Q$ as a $Q$-tuple and $\varphi(f)$ is 
the tuple of values of entries from $f$ under $\varphi$.  Thus, the left-hand side  of~\eqref{eq:CS_step} assumes the form
\[
	\Bigg|\sum_{\psi} 
	\prod_{f\in E(M_{\cQ_j})}\omega_{Q_{j+1}}\big(\psi(f_{Q_{j+1}})\big)
\sum_{\substack{\varphi\colon V(M_{\cQ_j})\to V\\ 
\varphi|_{V_{Q_{j+1}}(M_{\cQ_j})}\equiv\psi}}\prod_{f\in E(M_{\cQ_j})}\left(\indi_{\arrow E}\big(\varphi(f)\big)-d\right)
\bigg(\prod_{Q\in\cQ\setminus\cQ_{j+1}} \omega_Q\big(\varphi(f_Q)\big)\bigg)\Bigg|^{2}\!\!\,,
\]
where the first sum runs over all maps $\psi \colon V_{Q_{j+1}}(M_{\cQ_j})\to V$.

We then apply the Cauchy--Schwarz inequality with the product after the first sum forming the first sequence and the second sum forming the second sequence. The term $n^{|V_{Q_{j+1}}(M_{\cQ_j})|}$ on the right-hand side of~\eqref{eq:CS_step} comes from the first sequence after applying the Cauchy--Schwarz inequality and using  $\omega_{Q_{j+1}}^2\le 1$. Summing over the squares of the terms in the second sequence corresponds exactly to performing the doubling operation $\cl_{Q_{j+1}}$ -- the vertices outside of $V_{Q_{j+1}}(M_{\cQ_j})$ are doubled and all edges of $M_{\cQ_j}$ and their corresponding weight functions $\omega_Q$ are doubled as well. But this is exactly the sum on the right-hand side of~\eqref{eq:CS_step}, as required. 

Starting with~\eqref{eq:fdisc_cauchy} we apply~\eqref{eq:CS_step} $j=0,\ldots,\ell-1$  and obtain
\[
	\Bigg|\sum_{\bv\in V^{[k]}}\big(\indi_{\arrow E} (\bv)-d\big)\cW(\bv)\Bigg|^{2^\ell}  
	\le 
	\left(\prod_{j=0}^{\ell-1} \Big(n^{|V_{Q_{j+1}}(M_{\cQ_j})|}\Big)^{2^{\ell-j-1}}\right)\cdot
\left|\sum_{\varphi\colon V(M_{\cQ})\to V}\prod_{f\in E(M_{\cQ})}\left(\indi_{\arrow E}\big(\varphi(f)\big)-d\right)\right|\,.\label{eq:ktimes_CS}
\]
Owing to the  assumption~\eqref{eq:wdev_}, we arrive at
\begin{equation}\label{eq:wdisc_CS}
	\Bigg|\sum_{\bv\in V^{[k]}}\big(\indi_{\arrow E}(\bv)-d\big)\cW(\bv)\Bigg|^{2^\ell}
	\le 
	\bigg(\prod_{j=0}^{\ell-1} n^{|V_{Q_{j+1}}(M_{\cQ_j})|2^{\ell-j-1}}\bigg) \cdot \eps n^{v(M_\cQ)}\,.
\end{equation}

It remains to show that
\begin{equation}\label{eq:DEVDISCfinal}
	\sum_{j=0}^{\ell-1}2^{\ell-j-1}|V_{Q_{j+1}}(M_{\cQ_j})|+|V(M_\cQ)|=k2^\ell,
\end{equation}
since then the desired bound 
\[
\Bigg|\sum_{\bv\in V^{[k]}}\big(\indi_{\arrow E}(\bv)-d\big)\cW(\bv)\Bigg|\le \delta n^{k}\,.
\]
follows for $\eps=\delta^{2^\ell}$.

For the proof of~\eqref{eq:DEVDISCfinal} we observe that for every $i\in[k]$ and $j=0,\dots,\l$ 
we have 
\[
	\big|V_i(M_{\cQ_j})\big|=2^{j-\deg_{\cQ_j}(i)}\,,
\] 
since the $i$-th vertex of $K^{(k)}_k=M_\emptyset$ will be doubled for every edge of $Q\in\cQ_j$ with $i\not\in Q$.
Since~$\cQ=\cQ_\l$, we therefore obtain
\begin{align*}
\sum_{j=0}^{\ell-1}2^{\ell-j-1}|V_{Q_{j+1}}(M_{\cQ_j})|+|V(M_\cQ)|
& = \sum_{j=0}^{\ell-1}\sum_{i\in Q_{j+1}} 2^{\ell-1-\deg_{\cQ_j}(i)}  +\sum_{i=1}^k 2^{\ell-\deg_{\cQ}(i)} \\
& = \sum_{j=0}^{\ell-1}\sum_{i\in Q_{j+1}} 2^{\ell-\deg_{\cQ_{j+1}}(i)}  +\sum_{i=1}^k 2^{\ell-\deg_{\cQ}(i)}\\
& = \sum_{j=1}^{\ell}\sum_{i\in Q_{j}} 2^{\ell-\deg_{\cQ_{j}}(i)}  +\sum_{i=1}^k 2^{\ell-\deg_{\cQ}(i)}\\
& = \sum_{i\in \bigcup\cQ} \sum_{t=1}^{\deg_{\cQ}(i)} 2^{\ell-t}  +\sum_{i=1}^k 2^{\ell-\deg_{\cQ}(i)}\,.
\end{align*}
Viewing $\cQ$ as a (possibly non-uniform) hypergraph with vertex set $[k]$, we observe that  every 
isolated vertex~$i\in[k]\setminus\bigcup\cQ$ is not considered in the first double sum above and contributes 
$2^\ell$ to the second sum. Moreover, every vertex $i\in\bigcup\cQ$ contributes
\[
	2^\ell\Big(\frac{1}{2}+\frac{1}{4}+\dots+\frac{1}{2^{\deg_\cQ(i)}}\Big)+2^\ell\frac{1}{2^{\deg_\cQ(i)}}
	=2^\ell
\]
and, hence,~\eqref{eq:DEVDISCfinal} follows.
\end{proof}

\section{Distinguishing notions of quasirandomness}\label{sec:distinguish}
In this section we prove Proposition~\ref{lem:construction}, which roughly speaking asserts that the various notions of quasirandomness defined are distinct. We shall use the following notation and setup.
Let $V =[n]$ and order $V$ according to the natural ordering of $[n]$. For $\bv \in \binom{V}{k}$, we write $\bv^{(\mathrm{nat})}$ to denote the ordering of $v$ induced by the natural ordering of $[n]$. Then, given $Q \subseteq [k]$, we write $\bv_Q^{(\mathrm{nat})}$ to denote $(\bv^{(\mathrm{nat})})_Q$. Given $1 \leq i < k$ and a set $\cB \subseteq \binom{V}{i}$, we write~$H^{(k)}(\cB)$ to denote the $k$-uniform hypergraph whose vertex set is $V$ and where a set $\bv \in \binom{V}{k}$ is taken to be an edge of~$H^{(k)}(\cB)$ if the quantity 
$p_{\bv}= |\{\bv^{(\mathrm{nat})}_Q \in \arrow \cB\colon Q \in \cQ\}|$ is odd. The following lemma will facilitate the construction used in the proof of Proposition~\ref{lem:construction} below. 

\begin{lemma}\label{lem:B-exists}
For every $i \in [k-1]$ and $\eta >0$ there exists an $n_0$ such that, for every $n \geq n_0$, there is a set system $\cB \subseteq \binom{V}{i}$ with the following properties:
\begin{enumerate}[label=\rmlabel]
\item\label{itm:meet} 
For every sequence $\cG = (G_R)_{R \in \cR}$ of directed hypergraphs with $\cR \subseteq \powerset([i])$ having the property that $|R| < i$ for every $R \in \cR$,
\begin{equation}\label{eq:meet}
\big| \arrow \cB  \cap \cK_i(\cG) \big| = \tfrac{1}{2}|\cK_i(\cG)| \pm \eta n^i.
\end{equation}

\item\label{itm:density}
The edge density of $H^{(k)}(\cB)$ is $1/2 \pm \eta$.  

\item\label{itm:supp-dense}
If $\cF = (F_Q)_{Q \in \cQ}$ is the sequence of directed hypergraphs for $\cQ\subseteq \binom{[k]}{i}$ with $V(F_Q) = [n]$ and 
\[
	E(F_Q) = \big\{\bv \in V^{Q}\colon \bv^{(\mathrm{nat})} \notin \arrow \cB\big\}
\] 
for every $Q \in \cQ$, then 
$|\cK_k(\cF)| = (2^{-|\cQ|} \pm \eta) n^k$.
\end{enumerate}
\end{lemma}

\begin{proof}
We prove that a randomly chosen subset $\cB\subseteq \binom{V}{i}$ satisfies all of the above assertions with positive probability when $n$ is sufficiently large. Suppose then that $\cB \subseteq \binom{V}{i}$ is a random subset of the $i$-sets of $V$ where each $i$-set is placed in $\cB$ independently with probability $1/2$. 

To show that~\ref{itm:meet} holds with probability $1 - o(1)$, fix $\cG = (G_R)_{R \in \cR}$ subject to the restriction on~$\cR$ in~\ref{itm:meet}. The random variable $\big| \arrow \cB \cap \cK_i(\cG)\big|$ satisfies $\E\big[\big|\arrow \cB \cap \cK_i(\cG)\big|\big] = |\cK_i(\cG)|/2$. As 
\[
	\big|\arrow \cB \cap \cK_i(\cG)\big| = \sum_{\bv \in \cK_i(\cG)}\indi_{\arrow \cB}(\bv)
\] 
is a sum of independent indicator random variables (that is, $\indi_{\arrow \cB}(\bv)$ is equal to $1$ if $\bv \in \arrow \cB$ and zero otherwise), it follows, by Chernoff's inequality (see, e.g.,~\cite{JLR}*{Corollary~2.3}), that 
$$
\P \Big(\big| |\arrow \cB \cap \cK_i(\cG)| - |\cK_i(\cG)|/2\big| \geq \eta n^i  \Big) \leq 2^{-\Omega(n^i)}.
$$
As the number of possible sequences $\cG$ is $2^{O(n^{i-1})}$, it follows that $\cB$ satisfies the first property with probability $1-o(1)$ for $n$ sufficiently large. 

We proceed to~\ref{itm:density}. Suppose $H^{(k)}(\cB) = (V,E)$. For any $\bv \in \binom{V}{k}$, we have 
\[
	\P(\bv \in E) = \P (\text{\,$p_{\bv}$ is odd\,}) = \tfrac{1}{2}
\] 
and so $\E[|E|] = \frac{1}{2}\binom{n}{k}$. Writing $e(H^{(k)}(\cB)) = \sum_{\bv \in \binom{V}{k}}\indi_{E}(\bv)$, 
we see, by Chebyshev's inequality, that
\[
	\P\Big( \big||E| - \tfrac{1}{2}\tbinom{n}{k} \big| \geq  \eta \tbinom{n}{k} \Big) 
	\leq 
	\frac{\E[|E|]}{\big(2\eta \E[|E|]\big)^2} + \frac{\sum_{\bu,\bv \in \binom{V}{k}}\COV(\indi_E(\bu),\indi_E(\bv))}{\big(2\eta \E[|E|]\big)^2}\,,
\]
where the sum on the right-hand side ranges over $k$-sets $\bu$ and $\bv$ such that $\bu^{(\mathrm{nat})}_Q = \bv^{(\mathrm{nat})}_Q$ for some~$Q \in \cQ$. The number of such pairs of sets is $O(n^{2k-i})$. As $i \geq 1$ and $(\E[|E|])^2 = \Omega(n^{2k})$, it follows that $\cB$ satisfies the second property with probability $1- o(1)$ for $n$ sufficiently large. 

For the third property~\ref{itm:supp-dense}, note that $\bv \in \cK_k(\cF)$ if and only if $\bv^{(\mathrm{nat})}_Q \notin \arrow \cB$ for every $Q \in \cQ$. Therefore, $\E[|\cK_k(\cF)|] = 2^{-|\cQ|} n (n-1) \cdots (n-k+1)$. Concentration around this expectation may be established via the second moment method in a similar manner to the argument used for~\ref{itm:density}.
\end{proof}


Next we derive Proposition~\ref{lem:construction} from Lemma~\ref{lem:B-exists}.

\begin{proof}[Proof of Proposition~\ref{lem:construction}]
It suffices to verify the case when $\cU$ and $\cQ$ only differ by one $i$-element set and, 
without loss of generality, we will assume that $\cQ\setminus\cU=\{Q^*\}$ for 
\[
	Q^* = [k-i+1,k] = \{k-i+1,\dots,k\}\,.
\] 
Set $\delta =2^{-|\cQ| -3}$ and, given $\eps > 0$, set $\eta = \min \{\eps / 2^{|\cU|}, 2^{-|\cQ|-2} \}$. With this $i$ and $\eta$, let $n_0$ be the integer whose existence is guaranteed by Lemma~\ref{lem:B-exists} and, for every $n \geq n_0$, let $\cB_n \subseteq \binom{V}{i}$ be a set system satisfying the properties stipulated in that lemma. We consider  $H_n = H_n^{(k)}(\cB_n)$. 

By~\ref{itm:density} the density of $H_n$ is as required. To see that $\cH = (H_n)_{n \in \N}$ does not satisfy $\DISC_{\cQ,1/2}(\delta)$, let $\cF$ be as in~\ref{itm:supp-dense}. Then $\arrow E(H_n) \cap \cK_k(\cF)$ is the empty set, so 
$$
\big| |\arrow E(H_n) \cap \cK_k(\cF)| - 2^{-1} |\cK_k(\cF)| \big|  = 2^{-1} |\cK_k(\cF)| 
= \big(2^{-|\cQ|-1} \pm \eta \big) n^k 
\geq 2^{-|\cQ|-2} n^k > \delta n^k.
$$
It remains to show that $\cH$ satisfies $\DISC_{\cU,1/2}(\eps)$. To that end, fix a sequence of directed hypergraphs $\cG = (G_U)_{U \in \cU}$. Our aim is to prove that
$$
\big| \arrow E(H_n) \cap \cK_k(\cG)\big| = |\cK_k(\cG)|/2 \pm \eps n^k.
$$

Recall that $Q^*=[k-i+1,k]$. For $\bell \in V^{[k-i]}$ and $\bu \in V^{Q^*}$, we write $\bell \circ \bu$ to denote the member of $V^{[k]}$ 
satisfying $(\bell \circ \bu)_{[1,k-i]} = \bell$ and $(\bell \circ \bu)_{Q^*} = \bu$. Define
$$
\Ext(\bell) = \{\bu \in V^{Q^*}\colon \bell \circ \bu \in \arrow E(H_n) \cap \cK_k(\cG) \}
$$
to be the set of ways the $(k-i)$-tuple $\bell$ can be extended to a member of $\arrow E(H_n) \cap \cK_k(\cG)$. 
Then 
$$
\left|\arrow E(H_n) \cap \cK_k(\cG)\right| = \sum_{\bell \in V^{[k-i]}} |\Ext(\bell)|.
$$
A tuple $\bell \in V^{[k-i]}$ is said to {\em have potential for extension} if 
$\bell_U \in E(G_U)$ for every $U \in \cU$ not meeting~$Q^*$. Otherwise, we say $\bell$ has no potential. Observe that $|\Ext(\bell)| = 0$ if $\bell$ has no potential. In particular, we may write 
$$
\left|\arrow E(H_n) \cap \cK_k(\cG)\right| = \sum_{\bell \in \cP} |\Ext(\bell)|,
$$
where $\cP \subseteq V^{[k-i]}$ denotes all tuples that have potential for extension. To say more about $|\Ext(\bell)|$ for $\bell \in \cP$, we require some further notation.

We write $\cR(\bell)$ for the set of all $\bu \in V^{Q^*}$ such that $(\bell \circ \bu)_U \in E(G_U)$ for all $U \in \cU_{Q^*}$, where 
\[
	\cU_{Q^*} = \{U \in \cU\colon U \cap Q^* \not= \emptyset\}\,,
\] 
noting that $\bu \in V^{Q^*}$ cannot lie in $\Ext(\bell)$ unless it satisfies this condition. For each $G_U \in \cG$ with $U \in \cU_{Q^*}$, we define two directed hypergraphs. The first, $G_{U,\bell}^{\in}$, has $V$ as its vertex set and 
\[
	\big\{\bv_{U \cap Q^*}\colon \bv \in V^{[k]}\text{ with }\bv_{[1,k-i]} = \bell \tand \bv_U \in E(G_U) \cap \arrow \cB_n \big\}
\] 
for its (directed) edge set. The second, $G_{U,\bell}^{\notin}$, is defined similarly with $\arrow \cB_n$ replaced by its complement~$\arrow \cC_n$. That is, the vertex set of $G_{U,\bell}^{\notin}$ is $V$ and its edge set is 
\[
	\big\{\bv_{U \cap Q^*}\colon \bv \in V^{[k]} \text{ with } \bv_{[1,k-i]} = \bell \tand \bv_U \in E(G_U) \cap \arrow \cC_n\big\}\,.
\] 

In order to determine whether (a fixed) $\bu \in \cR(\bell)$ is in $\Ext(\bell)$, we consider three parameters:
\begin{enumerate}[label=\alabel]
\item The parity of the quantity $|\{\bell_U \in \arrow \cB_n \colon U \in \cU\}|$.
We write $p_{\bell}$ for this parity, treated as a residue modulo $2$, and refer to it as the parity of $\bell$.

\item The parity of the quantity 
\begin{align*}
	\Big|\big\{(\bell \circ \bu)_{U \cap Q^*} \in  E(G_{U,\bell}^{\in})\colon U \in \cU \tand U \cap Q^* \not= \emptyset\big\}\Big| 
 = \sum_{U \in \cU_{Q^*}} \indi_{E(G_{U,\bell}^{\in})}(\bu_{U \cap Q^*}) \label{eq:parity2}
\end{align*}
This is the parity of the number of $U \in \cU$ meeting $Q^*$ for which $(\bell \circ \bu)_{U}$ is supported by both $E(G_U)$ and $\arrow \cB_n$.  
We write $p'_{\bu}$ for this parity, again treated as a residue modulo $2$, and refer to it as the parity of $\bu$.  

\item The value of   (or, alternatively, $\indi_{\arrow \cC_n}(\bu)$). 
\end{enumerate}

Setting $p_{\bell,\bu} \equiv p_{\bell}+p'_{\bu} \; \mathrm{mod}\; 2$, we see that if $\bell \in \cP$ and $\bu \in V^{Q^*}$, then 
\begin{equation*}\label{eq:ext}
\indi_{\Ext(\bell)}(\bu) = 
\begin{cases}
1, &\text{if}\ \bu \in \cR(\bell)\; \text{and} \; p_{\bell,\bu} \not\equiv \indi_{\arrow \cB_n}(\bu) \; \mathrm{mod} \; 2,\\
0, &\text{if}\ \bu \in \cR(\bell) \; \text{and} \; p_{\bell,\bu} \equiv \indi_{\arrow \cB_n}(\bu)\; \mathrm{mod} \; 2,\\
0, &\text{if}\ \bu \notin \cR(\bell)\,.
\end{cases}
\end{equation*}
For instance, if $\bell \in \cP$ has even parity and $\bu \in \cR(\bell)$ has odd parity (so that $p_{\bell,\bu} \equiv 1 \; \mathrm{mod}\; 2$), then, in order to have $\bell \circ \bu \in \arrow E(H_n)$, one must have $\indi_{\arrow \cB_n}(\bu) = 0$ to attain the desired parity as per the definition of $H_n$. 
Therefore, for a fixed $\bell \in \cP$,
\begin{equation}\label{eq:parity3}
|\Ext(\bell)| = |\{\bu \in \cR(\bell)\colon p_{\bell,\bu} \not\equiv \indi_{\arrow \cB_n}(\bu)\; \mathrm{mod}\; 2\}|.
\end{equation}

The pairs $(G_{U,\bell}^{\in},G_{U,\bell}^{\notin})_{U \in \cU_{Q^*}}$ give rise to $2^{|\cU_{Q^*}|}$ sequences of directed hypergraphs. Enumerate these sequences arbitrarily and let $\cG_{j,\bell} = (G_U^{(j)})_{U \in \cU_{Q^*}}$ with $G_U^{(j)} \in \{G_{U,\bell}^{\in},G_{U,\bell}^{\notin}\}$, denote the $j$-th sequence in this enumeration. We shall refer to such sequences as \emph{signature sequences}. We say a signature sequence $\cG_{j,\bell}$ is {\it odd} if the number of its members appearing with the superscript $\in$ is odd. Otherwise, we say the sequence is {\it even}. In this way, each signature sequence is assigned a {\it parity}. 

Note now that for each $i$-tuple $\bu \in \cR(\bell)$ with parity $p'_{\bu}$ there exists a unique signature sequence~$\cG_{j,\bell}$ of the same parity such that $\bu  \in \cK_i(\cG_{j,\bell})$, given by taking 
\[
	G_U^{(j)} = 
	\begin{cases} 
		G_{U,\bell}^{\in},  & \text{if}\ \bu_{U \cap Q^*} \in E(G_{U,\bell}^{\in}),\\
		G_{U,\bell}^{\notin}, & \text{if}\ \bu_{U \cap Q^*} \in E(G_{U,\bell}^{\notin})\,.
	\end{cases}
\]
Therefore, since $\cK_i(\cG_{j,\bell}) \subseteq \cR(\bell)$ for each $j$, we see that the sets $\big(\cK_i(\cG_{j,\bell})\big)_{j=1}^{2^{|\cU_{Q^*}|}}$ form a partition of $\cR(\bell)$.

Given $\bell \in \cP$ and a signature sequence $\cG_{j,\bell}$ of parity $p$, we set 
$$
f(\bell,\cG_{j,\bell}) = 
\begin{cases}
\arrow \cB_n, & \text{if}\ p_{\bell} + p \equiv 0 \; \mathrm{mod}\; 2\,\\
\arrow \cC_n, &  \text{if}\ p_{\bell} + p \equiv 1 \; \mathrm{mod}\; 2\,.
\end{cases}
$$
By the discussion above, we may then rewrite~\eqref{eq:parity3} as
$$
|\Ext(\bell)| = \sum_{j=1}^{2^{|\cU_{Q^*}|}} \big|f(\bell,\cG_{j,\bell}) \cap \cK_i(\cG_{j,\bell})\big|,
$$ 
which in turn yields
\begin{equation}\label{eq:sequences}
|\arrow{E}(H_n) \cap \cK_k(\cG)| = \sum_{\bell \in \cP} \sum_{j=1}^{2^{|\cU_{Q^*}|}} \big|f(\bell,\cG_{j,\bell}) \cap \cK_i(\cG_{j,\bell})\big|.\end{equation}

We now claim that  
\begin{equation}\label{eq:support}
|\cK_k(\cG)| = \sum_{\bell \in \cP} \sum_{j=1}^{2^{|\cU_{Q^*}|}} |\cK_i(\cG_{j,\bell})|.
\end{equation}
To see this, fix $\bv \in \cK_k(\cG)$ and write $\bv = \bell \circ \bu$ where $\bv_{[k-i+1]} = \bell$ and $\bv_{Q^*} = \bu$. For such a~$\bv$, we have $\bv_U \in E(G_U)$ for every $U \in \cU$, so that $\bell \in \cP$ and $\bu \in \cR(\bell)$. The inclusion of the members of the sequence $(\bv_U)_{U \in \cU_{Q^*}}$ in $\arrow \cB_n$ or $\arrow \cC_n$ defines a unique signature sequence (with respect to $\bell$), namely, $\cG_{j^*,\bell}$ for some appropriate $j^*$, such that $\bu \in \cK_i(\cG_{j^*,\bell})$. Indeed, $\bv_U = (\bell \circ \bu)_U \in E(G_U)$ for each $U \in \cU_{Q^*}$, so that $(\bell \circ \bu)_U \in \arrow \cB_n$ implies that $\bu_{U \cap Q^*} \in E(G^{\in}_{U,\bell})$ and $(\bell \circ \bu)_U \in \arrow \cC_n$ implies that $\bu_{U \cap Q^*}\in  E(G^{\notin}_{U,\bell})$. Therefore, every $\bv \in \cK_k(\cG)$ can be written as $\bell\circ \bu$ with $\bell \in \cP$ and $\bu \in \cK_i(\cG_{j^*,\bell})$ for some $j^*$. Conversely, given $\bell \in \cP$ and $\bu \in \cK_i(\cG_{j,\bell}) \subseteq \cR(\bell)$ for some $j$,  the tuple~$\bell \circ \bu $ automatically satisfies $(\bell \circ \bu)_{U} \in E(G_U)$ for every $U \in \cU$. The claim then follows.

Returning to~\eqref{eq:sequences}, we see that 
\begin{multline*}
	\big| \arrow E(H_n) \cap \cK_k(\cG) \big| 
	= 
	\sum_{\bell \in \cP} \sum_{j=1}^{2^{|\cU_{Q^*}|}} \big| f(\bell,j) \cap \cK_i(\cG_{j,\bell}) \big| 
	\overset{\eqref{eq:meet}}{=} 
	\sum_{\bell \in \cP} \sum_{j=1}^{2^{|\cU_{Q^*}|}} \big( |\cK_i(\cG_{j,\bell})|/2 \pm \eta n^i \big) \\
	= 
	\frac{1}{2} \sum_{\bell \in \cP} \sum_{j=1}^{2^{|\cU_{Q^*}|}} |\cK_i(\cG_{j,\bell})| 
		\pm 2^{|\cU|} \eta \sum_{\bell \in V^{k-i}}  n^i
\overset{\eqref{eq:support}}{=} \frac{|\cK_k(\cG)|}{2} \pm \eps n^k,
\end{multline*}
as required.
\end{proof}

\vspace{3mm}
\noindent
\textbf{Acknowledgements.} We are indebted to the anonymous referee for their careful review.

\bibliographystyle{amsplain}
\bibliography{quasi-hyper}

\end{document}